\documentclass[10pt,a4paper]{article}

\usepackage{a4wide}
\usepackage{amsthm}
\usepackage{amsfonts}
\usepackage{amsmath}
\usepackage[english]{babel}

\usepackage{faktor}
\usepackage{etex}
\usepackage{hyperref}
\usepackage{multimedia}
\usepackage[all]{xy}
\usepackage{amssymb}
\usepackage{graphicx,textpos}
\usepackage[dvipsnames]{xcolor}
\usepackage{helvet}
\usepackage{stmaryrd}
\usepackage{enumerate}
\usepackage{todonotes}
\usepackage{pdfsync}
\usepackage{dsfont}
\usepackage[shortcuts]{extdash}
\usepackage[all]{xy}
  \newdir{ >}{{}*!/-9pt/@{>}}

\newcommand{\N}{\mathbb{N}}
\newcommand{\R}{\mathbb{R}}
\newcommand{\Q}{\mathbb{Q}}
\newcommand{\Z}{\mathbb{Z}}
\newcommand{\C}{\mathbb{C}}

\newcommand{\T}{\mathbb{T}}
\newcommand{\rst}[1]{\ensuremath{{\mathbin\mid}\raise-.5ex\hbox{$#1$}}}

\newcommand{\lien}{\mathfrak{n}}
\newcommand{\lieh}{\mathfrak{h}}
\newcommand{\liem}{\mathfrak{m}}
\newcommand{\lies}{\mathfrak{so}}

\DeclareMathOperator{\trace}{Tr}
\DeclareMathOperator{\Gal}{Gal}
\DeclareMathOperator{\GL}{GL}
\DeclareMathOperator{\SL}{SL}
\DeclareMathOperator{\Aut}{Aut}

\DeclareMathOperator{\sg}{sgn}
\DeclareMathOperator{\Pf}{Pf}

\author{Jonas Der\'e\thanks{The author was supported by a Ph.D.~fellowship of the Research Foundation -- Flanders (FWO).
Research supported by the research Fund of the KU Leuven}\\
KU Leuven Kulak, E. Sabbelaan 53, BE-8500 Kortrijk, Belgium}
\title{\bf A new method for constructing Anosov Lie algebras}
\date{}

\newtheorem{Def}{Definition}[section]
\newtheorem*{Ex}{Example}
\newtheorem{Cor}[Def]{Corollary}
\newtheorem{Thm}[Def]{Theorem}
\newtheorem{Prop}[Def]{Proposition}
\newtheorem{Lem}[Def]{Lemma}
\newtheorem{Rmk}[Def]{Remark}
\newtheorem*{Prop*}{Proposition}
\newtheorem*{Lem*}{Lemma}
\newtheorem{CL}[Def]{(False) Claim}
\newtheorem{QN}{Question}
\newtheorem*{Con}{Conjecture}
\begin{document}

\maketitle

\begin{abstract}
It is conjectured that every closed manifold admitting an Anosov diffeomorphism is, up to homeomorphism, finitely covered by a nilmanifold. 
Motivated by this conjecture, an important problem is to determine which nilmanifolds admit an Anosov diffeomorphism. The main theorem of this article gives a general method for constructing Anosov diffeomorphisms on nilmanifolds. As a consequence, we give new examples which were overlooked in a corollary of the classification of low-dimensional nilmanifolds with Anosov diffeomorphisms and a correction to this statement is proven. This method also answers some open questions about the existence of Anosov diffeomorphisms which are minimal in some sense.
\end{abstract}

A diffeomorphism $f: M \to M$ on a closed manifold is called Anosov if the tangent bundle splits continuously into two $df$-invariant vector bundles $E^s$ and $E^u$, such that $df$ is contracting on $E^s$ and expanding on $E^u$. The standard examples of Anosov diffeomorphisms are induced by unimodular hyperbolic automorphisms of $\R^n$ on the $n$-torus $\T^n$, considered as quotient space $\faktor{\R^n}{\Z^n}$. The first non-toral example was given by S. Smale in his paper \cite{smal67-1} where he also raised the question of classifying all closed manifolds admitting an Anosov diffeomorphism. 

It is conjectured that every Anosov diffeomorphism is topologically conjugate to an affine infra-nilmanifold automorphism. A proof of this conjecture would thus imply that every closed manifold admitting an Anosov diffeomorphism is homeomorphic to an infra-nilmanifold admitting an Anosov diffeomorphism. More details about the definition and the conjecture can be found in \cite{deki11-1,smal67-1}. In recent years, there has been quite some research about classifying infra-nilmanifolds supporting an Anosov diffeomorphism as well as constructing new examples with specific properties.

The conjecture motivates the study of infra-nilmanifolds supporting an Anosov diffeomorphism and in this paper we give a general method for constructing examples in the case of nilmanifolds. The existence of an Anosov diffeomorphism on a nilmanifold is equivalent to the existence of an Anosov automorphism on the corresponding nilpotent Lie algebra, i.e.\ a hyperbolic and integer-like automorphism (see Section \ref{intro} for more details). One way to construct Anosov automorphisms on Lie algebras is to start from an automorphism of a free nilpotent Lie algebra and take a quotient by an ideal which is invariant under the automorphism, see for example the papers \cite{dm05-1,dd03-2,frie81-1, payn09-1}. Another type of construction, as was used in \cite{laur03-1,lw08-1,lw09-1,mw07-1}, starts from a nilpotent Lie algebra $\lien^E$ over some field extension $E$ of $\Q$ with a hyperbolic automorphism and then gives a rational form of the Lie algebra which is invariant under the automorphism. The hard step in the latter case is to construct an explicit basis of $\lien^E$ such that the structure constants are rational and the matrix representation of the automorphism has rational entries. Note that in these papers, it is checked for each example separately that the given set of vectors is in fact a basis, that the structure constants do lie in $\Q$ and that the matrix of $f$ with respect to this basis has entries in $\Q$. The computations needed for these steps are rather cumbersome and somewhat time-consuming and these can be avoided by the main theorem of this paper.

The main theorem of this article generalizes this second method and states that given a hyperbolic and integer-like automorphism of a Lie algebra $\lien^E$ which behaves `nicely' under the action of $\Gal(E,\Q)$ (see Section \ref{secGC} for the exact statement), there always exists a rational form of $\lien^E$ which is an Anosov Lie algebra. Not only does this technique significantly shorten the construction of the examples in \cite{laur03-1,lw08-1,lw09-1,mw07-1}, it also allows us to give new interesting examples of Anosov diffeomorphisms on nilmanifolds. For example, we can give a positive answer to some existence questions stated in \cite{laur03-1,lw09-1}. 

The construction from the main theorem is also important for the classification of infra-nilmanifolds supporting an Anosov diffeomorphism. The only general classification is given by Porteous in \cite{port72-1} for flat manifolds and this was generalized to infra-nilmanifolds modeled on free nilpotent Lie groups in \cite{dd13-1}, based on the work of \cite{dv09-1,dv11-1}. For nilmanifolds, there is a classification of Anosov automorphisms on Lie algebras up to dimension $8$ in \cite{lw09-1}. One of the consequences of the main theorem is the construction of some new examples which were overlooked in this classification and we formulate a correction to this statement in Section \ref{verb}. 

This article is built up as follows. In the first section, we give the main definitions about Anosov diffeomorphisms on nilmanifolds and state the correspondence between Anosov diffeomorphisms and hyperbolic, integer-like automorphisms of rational Lie algebras. The second part discusses the exact statement and proof of the main theorem. The last part describes the consequences of the main theorem, including the new examples answering the questions mentioned above.

\section{Anosov diffeomorphisms on nilmanifolds}
\label{intro}

We start with recalling the definitions and main results about Anosov diffeomorphisms on nilmanifolds. We also introduce the signature of an Anosov diffeomorphism and the type of a nilpotent Lie algebra, including some existence questions about Anosov automorphisms we will answer in this paper. 

\smallskip

Let $N$ be a connected, simply connected, nilpotent Lie group with a lattice $\Gamma$, i.e.\ a discrete subgroup $\Gamma \subset N$ with compact quotient $\Gamma \backslash N$. The quotient space $\Gamma \backslash N$ is a closed manifold with fundamental group $\Gamma$ and is called a nilmanifold. The group $\Gamma$ is nilpotent, finitely generated and torsion-free and every group satisfying those three properties occurs as the fundamental group of a nilmanifold (see e.g. \cite{deki96-1}). Every Lie group automorphism $\alpha: N \to N$ induces a Lie algebra automorphism on the Lie algebra of $N$ and the eigenvalues of $\alpha$ are defined as the eigenvalues of this Lie algebra automorphism. If $\alpha \in \Aut(N)$ is a hyperbolic automorphism, meaning that it has no eigenvalues of absolute value $1$, and if $\alpha(\Gamma) = \Gamma$, then $\alpha$ induces an Anosov diffeomorphism on the nilmanifold $\Gamma \backslash N$ and the induced map is called a hyperbolic nilmanifold automorphism. For more details and the general definitions of infra-nilmanifolds and affine infra-nilmanifold automorphisms we refer to the paper \cite{deki11-1}.

Let $\Gamma \backslash N$ be a nilmanifold and denote by $\Gamma^\Q$ the radicable hull of $\Gamma$ (also called rational Mal'cev completion). Corresponding to this radicable hull $\Gamma^\Q$, there is also a rational Lie algebra $\lien^\Q$ and every finite dimensional rational nilpotent Lie algebra can be found in this way (see e.g. \cite{sega83-1}). Two lattices $\Gamma_1$ and $\Gamma_2$ of $N$ are called commensurable if there exist finite index subgroups $\Gamma_1^\prime \le \Gamma_1$ and $\Gamma_2^\prime \le \Gamma_2$ such that $\Gamma_1^\prime$ and $\Gamma_2^\prime$ are isomorphic. This is equivalent to the property that the Lie algebras of the radicable hulls $\Gamma_1^\Q$ and $\Gamma_2^\Q$ are isomorphic.

\smallskip

From \cite[Corollary 3.5.]{deki99-1} and \cite{mann74-1} it follows that a nilmanifold $\Gamma \backslash N$ admits an Anosov diffeomorphism if and only if there exists a hyperbolic and integer-like automorphism of $\lien^\Q$. Recall that a matrix is called integer-like if its characteristic polynomial has coefficients in $\Z$ and its determinant has absolute value $1$.  Since the eigenvalues and characteristic polynomial of a matrix are invariant under conjugation, the properties hyperbolic and integer-like are invariant under change of basis. An automorphism is hyperbolic respectively integer-like if the matrix representation of the automorphism for some (and thus for every) basis has the same property. This motivates the following definition:

\begin{Def}
An automorphism $\alpha \in \Aut(\lien^E)$ of a Lie algebra $\lien^E$ over some field $E \subseteq \C$ is called \emph{Anosov} if it is hyperbolic and integer-like. A rational Lie algebra $\lien^\Q$ with an Anosov automorphism is called \emph{Anosov}.
\end{Def}
\noindent The superscripts $^\Q$ and $^E$ indicate over which field we are working. Thus a classification of all nilmanifolds admitting an Anosov diffeomorphism is equivalent to a classification of all Anosov Lie algebras. A theorem by Jacobson (see \cite{jaco55-1}) shows that every Anosov Lie algebra is necessarily nilpotent. 

\smallskip

Theorem \ref{main} will give us a very general way of constructing Anosov Lie algebras. By using the low-dimensional classification of Anosov Lie algebras, we can describe the isomorphism class for some of those Lie algebras in Section \ref{verb}, but it will be hard to do the same in general. On the other hand, there will be some properties of the Lie algebra or the Anosov automorphisms which follow immediately from the construction, e.g.\ the signature of the Anosov automorphism and the type of the Lie algebra, properties which we introduce below.

\smallskip

The definition of the signature of an Anosov diffeomorphism $f: M \to M$ makes use of the $df$-invariant splitting $TM = E^s \oplus E^u$:
\begin{Def}
The signature $\sg(f)$ of an Anosov diffeomorphism $f: M \to M$ is defined as the set $\{\dim_\R E^s, \dim_\R E^u\}$.
\end{Def}
\noindent  We define the signature as a set rather than as an ordered pair since the order of the elements does not play a role. Indeed, the inverse of an Anosov diffeomorphism is again Anosov with the bundles $E^s$ and $E^u$ interchanged. For Anosov automorphisms $f: \lien^\Q \to \lien^\Q$ on a rational Lie algebra $\lien^\Q$, the signature is defined as the set $\{p,q\}$ where $p$ is the number of eigenvalues with absolute value $<1$ and $q$ the number of eigenvalues of absolute value $> 1$. Note that the proof of \cite[Corollary 3.5.]{deki99-1} implies that there exists an Anosov diffeomorphism $f$ with signature $\sg(f)$ if and only if there exists an Anosov automorphisms with signature $\sg(f)$ on the corresponding Lie algebra. 

The following question is stated in \cite{lw09-1}:
\begin{QN}
\label{q1}
Does there exist an Anosov automorphism on a non-abelian Lie algebra with signature $\{2,k\}$ for some $k \in \N$?
\end{QN}
\noindent In Section \ref{sign} it is shown that $\min(\sg(f)) \geq c$ where $c$ is the nilpotency class of the Lie algebra. So a more general question is the existence of Anosov automorphisms which attain this lower bound for the signature. The main theorem gives us a positive answer in Section \ref{sign}.

As a consequence of a low-dimensional classification of Anosov automorphisms, the following is stated as a corollary in \cite{lw09-1}:
\begin{CL}
\label{corlaur}
Let $\Gamma \backslash N$ be a nilmanifold of dimension $\leq 8$ which admits an Anosov diffeomorphism. Then $\Gamma \backslash N$ is a torus or the dimension is $6$ or $8$ and the signature is $\{3,3\}$ or $\{4,4\}$ respectively.
\end{CL}
\noindent Examples illustrating that some cases were overlooked in this corollary are given in Section \ref{verb} with signature $\{2,4\}$ and $\{3,5\}$. A complete list of Anosov Lie algebras admitting such signatures is given and thus this section forms a correction to the result of \cite{lw09-1}.

\smallskip

If $\lien^E$ is a Lie algebra over a field $E$, then its lower central series $\gamma_i(\lien^E)$ is recursively defined by $\gamma_1(\lien^E) = \lien^E$ and $\gamma_i(\lien^E) = [\lien^E, \gamma_{i-1}(\lien^E) ]$ for all $i \geq 2$. The Lie algebra $\lien^E$ is said to be $c$-step nilpotent (or to have nilpotency class $c$) if $\gamma_{c+1}(\lien^E) = 0$ and $\gamma_{c}(\lien^E) \neq 0$. The type of a Lie algebra gives us information about the quotients of the lower central series:

\begin{Def}
The type of a nilpotent Lie algebra $\lien^E$ of nilpotency class $c$ is defined as the $c$-tuple $(n_1,\ldots,n_c)$, where $n_i = \dim_E \faktor{\gamma_{i}(\lien^E)}{\gamma_{i+1}(\lien^E)}$.
\end{Def}
\noindent If $\lien^\Q$ is an Anosov Lie algebra of type $(3, n_2, \ldots,n_c)$, then \cite[Theorem 1.3.]{payn09-1} states that $3 \mid n_i$ for every $i \in \{2, \ldots, c\}$. Combining this result with \cite[Proposition 2.3.]{lw08-1} we get that every Anosov Lie algebra $\lien^\Q$ of type $(n_1,\ldots,n_c)$ satisfies one of the following:
\begin{enumerate}[(i)]
\item $\lien^\Q$ is abelian or
\item $n_1 \geq 4$ and $n_i \geq 2$ for all $i \in \{2, \ldots, c\}$ or
\item $n_1 = n_2 = 3$ and $3 \mid n_i $ for all $i \in \{3, \ldots, c\}$.
\end{enumerate}
In \cite{laur03-1} it was proved that the lower bound of (ii) occurs for every $c$ and the main theorem shows that also the lower bound of (iii) is attained for every $c$, answering the following question of \cite{lw08-1}:

\begin{QN}
\label{q2}
Does there exist a $c$-step Lie algebra of type $(3, \ldots, 3)$ of nilpotency class $c \geq 3$ which is Anosov?
\end{QN}
\noindent An Anosov Lie algebra which attains one of the lower bounds (ii) or (iii) for the type will be called of minimal type. In fact, in Section \ref{sign} we will show that Question \ref{q2} has a positive answer even if we replace $3$ by any integer $n > 2$. 

\smallskip

To end this first section, we give a short summary of the known results about nilmanifolds supporting an Anosov diffeomorphism. In a few cases, a complete classification is given, for example in low-dimensional cases (see \cite{lw08-1}) and in the case of free nilpotent Lie algebras (see \cite{dani03-1,dd13-1}). 
In \cite{dm05-1,dd03-2,frie81-1, payn09-1}, some new examples of Anosov diffeomorphisms on nilmanifolds were constructed by taking quotients of free nilpotent Lie algebras. This method is based on the work of L. Auslander and J. Scheuneman in \cite{as70-1}. All other examples are explicitly constructed and the main result of this article is to generalize this construction in Theorem \ref{main}. This theorem avoids all the computations and allows us to construct more complicated examples of Anosov automorphisms. It seems plausible that all questions about the existence of Anosov automorphisms of specific signature on a Lie algebra of specific type can be answered in this way. The disadvantage of this method is that in general it is hard to describe the rational Lie algebra in terms of a basis and relations on this basis.

\section{Construction of Anosov Lie algebras}
\label{con}

The examples of Anosov Lie algebras in \cite{laur03-1,lw09-1,mw07-1} are constructed from a nilpotent Lie algebra over some field $E$ given by its decomposition into eigenspaces of an Anosov automorphism. An explicit basis is then constructed which is `symmetric' under the action of the Galois group of $E$. Instead of working with a basis, we rather use this `symmetric' property of the basis as the definition of a rational form of the Lie algebra. We will make this statement precise in Section \ref{secGC}. For automorphisms, a similar `symmetric' property can be defined such that the automorphism induces an automorphism on the rational form. In this way, no explicit calculations are needed and Theorem \ref{main} allows us to easily show the existence of several new Anosov Lie algebras. This method further generalizes the work of \cite{payn09-1} which starts from polynomials instead of field extensions.

\smallskip

First let us fix some notations for this section. Let $E$ be a subfield of $\C$ such that $E$ has finite degree over $\Q$. If $n$ is the extension degree of $E \supseteq \Q$, then $E$ is called a Galois extension of $\Q$ if the group $$\Gal(E,\Q) = \{ \sigma: E \to E \mid \sigma \text{ is a field automorphism} \}$$ has order $n$. We will always assume that our field extensions are Galois. 

A vector space $V$ over $E$ will be denoted by $V^E$ to indicate over which field we are working, as we already did above in Section \ref{intro}. All the vector spaces and Lie algebras we consider are finite dimensional. If $F \supseteq E$ is a field extension, then we can consider the vector space $F \otimes_E V^E$ which we will denote as $V^F$. If $G$ is a finite group and $\rho: G \to \GL(V^E)$ is a representation, then there also exists a representation $\rho^F: G \to \GL(V^F)$ by extending the scalars, i.e.\ by considering $\GL(V^E)$ as a subgroup of $\GL(V^F)$. The same notations will be used for Lie algebras, e.g.\ a Lie algebra over the rationals $\Q$ will be denoted by $\lien^\Q$. We say that two representations $\rho_1: G \to \GL(V^E)$ and $\rho_2: G \to \GL(W^E)$ are $E$-equivalent if there exists an isomorphism $\varphi: V^E \to W^E$ such that $\rho_2(g) \circ \varphi = \varphi \circ \rho_1(g)$ for all $g \in G$.

A rational subspace $W^\Q \subseteq V^E$ is called a rational form if some (and hence every) basis of $W^\Q$ over $\Q$ is also a basis of $V^E$ over $E$. If $\lien^E$ is a Lie algebra over $E$, we call a rational subalgebra $\liem^\Q \subseteq \lien^E$ a rational form if it is a rational form seen as subspace of $\lien^E$ as vector space. 

\subsection{Construction of a rational form}
\label{secGC}

Instead of focusing on the basis of the rational form of a vector space, we focus on the defining property of the rational form as being `symmetric' under the action of the Galois group. For this we first introduce the action of the Galois group on a vector space and define for each representation a rational subspace.

\smallskip

Consider the natural right action of $\Gal(E,\Q)$ on the field $E$, given by $$ \forall \sigma \in \Gal(E,\Q), \forall x \in E:  x^\sigma = \sigma^{-1}(x).$$ 
By defining the action component-wise, there is also a natural right action on the vector space $E^m$. Note that the relations
\begin{equation}
\label{relatie}
\begin{aligned}[c]
\left(\sigma(\lambda) x\right)^\sigma = \lambda x^\sigma
\end{aligned}
\text{ and }
\begin{aligned}[c]
\left(x + y\right)^\sigma = x^\sigma + y^\sigma
\end{aligned}
\end{equation}
hold for all $x,y \in E^m$, $\sigma \in \Gal(E,\Q)$ and $\lambda \in E$. Let $\rho: \Gal(E,\Q) \to \GL_m(E)$ be a representation, then it follows immediately that the subset defined as \begin{align}
V^\Q_\rho = \{ v  \in E^m \mid \forall \sigma \in \Gal(E,\Q), \rho_\sigma(v) = v^\sigma \} \label{def}
\end{align} is a rational subspace of $E^m$. This subspace is already close to being a rational form of $E^m$ in the sense of the following lemma:

\begin{Lem}
\label{ratform}
If a set of vectors $v_1, \ldots, v_k$ of $V^\Q_\rho$ is linearly independent over $\Q$, then this set is also linearly independent over $E$ as vectors of $E^m$.
\end{Lem}

\begin{proof}
Assume that the lemma does not hold and take vectors $v_1, \ldots, v_k$ of $V^\Q_\rho$ with $k$ minimal which are linearly independent over $\Q$ but contradict the statement. This means that there exists $x_i \in E$ such that
\begin{align}\label{linind} \sum_{i=1}^k x_i v_i = 0.\end{align} From the minimality of $k$ it follows that $x_1 \neq 0$ and thus by multiplying this equation by $x_1^{-1}$ we can assume that $x_1 = 1$. Since $k$ is minimal, it follows that the $x_i$ are the unique elements of $E$ such that $x_1 = 1$ and equation (\ref{linind}) is true. 

\smallskip

If we apply the map $\rho_\sigma$ to the equation, we get that $$0 = \rho_\sigma\left(\sum_{i=1}^k x_i v_i \right) = \sum_{i=1}^k x_i \rho_\sigma(v_i) = \sum_{i=1}^k x_i v_i^\sigma = \left(\sum_{i=1}^k \sigma(x_i) v_i\right)^\sigma$$ because of (\ref{relatie}). We also have that $$\sum_{i=1}^k \sigma(x_i) v_i = 0.$$ Minimality of $k$ and the fact that $\sigma(x_1) = \sigma(1) = 1$ imply that $\sigma(x_i) = x_i$ for all $i$. Because this statement holds for all $\sigma \in \Gal(E,\Q)$, we conclude that the coefficients $x_i$ lie in $\Q$ and thus we get a contradiction since $v_i$ was a set of linearly independent vectors over $\Q$. 
\end{proof}

The lemma shows that the rational subspace $V^\Q_\rho$ is a rational form of $E^m$ if its dimension is maximal. This motivates the following definition:
\begin{Def}
A representation $\rho: \Gal(E,\Q) \to \GL_m(E)$ is called \emph{Galois compatible} (abbreviated as GC) if and only if $\dim_\Q (V^\Q_\rho) = m$. Equivalently, the representation $\rho
$ is GC if and only if $E \otimes V^\Q_\rho = E^m$ or if $V^\Q_\rho $ is a rational form of $E^m$.
\end{Def}

The trivial representation is the easiest example of a GC representation, since $V^\Q_\rho = \Q^m$. A simple computation shows that also the regular representation is GC:

\begin{Ex}
\label{reg}\emph{
Let $\rho$ be the regular representation of $\Gal(E,\Q)$, i.e.\ take the vector space over $E$ spanned by the basis $\{v_\sigma \mid \sigma \in \Gal(E,\Q)\}$ and the representation $\rho$ induced by the relations $\rho_\tau(v_\sigma) = v_{\tau\sigma}$ for all $\tau, \sigma \in \Gal(E,\Q)$. Every element $v$ of the rational vector space $V^\Q_\rho$ is given by $$v = \sum_{\sigma \in \Gal(E,\Q)} \sigma(x) v_\sigma$$ for some $x \in E$. This is a rational vector space of dimension $[E:\Q]$, which is also the order of the group $\Gal(E,\Q)$ and thus the dimension of the regular representation $\rho$. This shows that $\rho$ is indeed GC.}
\end{Ex}
In a similar way it is easy to check that every representation which is given by permutation matrices is GC. The goal of the remaining part of this section is to show that every rational representation is GC:
\begin{Prop}
\label{GC}
Let $E$ be a Galois extension of $\Q$. Every representation $\Gal(E,\Q) \to \GL_m(\Q)$ is Galois compatible.
\end{Prop}

\noindent To prove this statement we will first show that it holds for irreducible representations and then apply the following lemma:

\begin{Lem}
\label{som}
Let $\rho_i: \Gal(E,\Q) \to \GL_{n_i}(E)$ with $i \in \{1,2\}$ be representations, then the following are equivalent:
\begin{enumerate}
\item $\rho_1$ and $\rho_2$ are Galois compatible.
\item $\rho_1 \oplus \rho_2: \Gal(E,\Q) \to \GL_{n_1+n_2}(E)$ is Galois compatible.
\end{enumerate}
\end{Lem}

\begin{proof}
Let $V_{\rho_1}^\Q$ and $V_{\rho_2}^\Q$ be the rational subspaces corresponding to $\rho_1$ and $\rho_2$, then by defition it follows that $$V_{\rho_1 \oplus \rho_2}^\Q = V_{\rho_1}^\Q \oplus V_{\rho_2}^\Q.$$ The statement of the lemma then easily follows by using Lemma \ref{ratform}.
\end{proof}
\noindent The proof of the proposition is now immediate by using the previous lemma:
\begin{proof}[Proof of Proposition \ref{GC}]
If two representation $\rho_1, \rho_2: \Gal(E,\Q) \to \GL_m(\Q)$ are $\Q$-equivalent, then $\rho_1$ is GC if and only if $\rho_2$ is GC. So from Lemma \ref{som} it follows that it is sufficient to prove the proposition for the $\Q$-irreducible representations of $\Gal(E,\Q)$. Since every $\Q$-irreducible representation is a subrepresentation of the regular representation (see e.g.\ \cite[Corollary 9.5.]{isaa76-1}), the statement follows from the example above.
\end{proof}
%

\subsection{Construction of automorphisms on rational forms of Lie algebras}
\label{cons}
In the previous subsection we constructed rational forms of vector spaces $E^m$ from representations of the Galois group $\Gal(E,\Q)$. In this subsection we use this technique to construct Lie algebras and automorphisms on them, allowing us to construct Anosov Lie algebras as well.

\smallskip

Let $\lien^\Q$ be a rational Lie algebra and $E$ a Galois extension of $\Q$. By taking a basis for $\lien^\Q \subseteq \lien^E$, the vector space $\lien^E$ is isomorphic to $E^m$ with $m = \dim_\Q(\lien^\Q)$ and thus the right action of $\Gal(E,\Q)$ is defined on $\lien^E$. It is an exercise to check that this action does not depend on the choice of basis for $\lien^\Q$. The right action of $\Gal(E,\Q)$ satisfies the property $$[X,Y]^\sigma = [ X^\sigma,Y^\sigma]$$ for all $\sigma \in \Gal(E,\Q)$, $X,Y \in \lien^E$. For every representation $$\rho: \Gal(E,\Q)\to \Aut(\lien^\Q) \le \GL(\lien^\Q),$$ with $\GL(\lien^\Q)$ the isomorphisms of $\lien^\Q$ as vector space, it follows from the definition that the rational subspace $V^\Q_\rho$ forms a subalgebra of $\lien^E$ and is thus a rational form of $\lien^E$ by Proposition \ref{GC}. We will denote this subspace by $\liem^\Q_\rho$ to emphasize the fact that it is a subalgebra.

\smallskip

We now answer the question when a given automorphism $f: \lien^E \to \lien^E$ induces an automorphism on $\liem^\Q_\rho$. Note that $\Gal(E,\Q)$ also acts on the right on $\Aut(\lien^E)$, by defining for all $f \in \Aut(\lien^E)$ the action as $$ f^\sigma(v) = \left( f(v^{\sigma^{-1}}) \right)^\sigma.$$ By fixing a basis for $\lien^\Q$ and thus also for $\lien^E$, the matrix representation of $f^\sigma$ is given by applying $\sigma^{-1}$ to every entry of the matrix representation of $f$. Also every representation $\rho: \Gal(E,\Q) \to \Aut(\lien^E)$ induces a left action on $\Aut(\lien^E)$ by conjugation. The automorphisms where this left action corresponds with the right action are exactly those that induce an automorphism on $\liem^\Q_\rho$:

\begin{Lem}
\label{auto}
Let $\lien^\Q$ be a rational Lie algebra and $\rho: G \to \Aut(\lien^\Q)$ a representation. An element $f \in \Aut(\lien^E)$ induces an automorphism on $\liem^\Q_\rho$ if and only if $f^\sigma = \rho_\sigma f \rho_{\sigma^{-1}}$ for all $\sigma \in \Gal(E,\Q)$.
\end{Lem}
\begin{proof}
First assume that $f$ satisfies the condition, i.e.\ that $f^\sigma = \rho_\sigma f \rho_{\sigma^{-1}}$ for $\sigma \in \Gal(E,\Q)$. For every $v \in \liem^\Q_\rho$ and $\sigma \in \Gal(E,\Q)$, we have that 
\begin{align*}
\big(f(v)\big)^\sigma = f^\sigma (v^\sigma) = \rho_\sigma f \rho_{\sigma^{-1}} (\rho_\sigma(v)) = \rho_\sigma(f(v))
\end{align*}
and thus $f(v) \in \liem^\Q_\rho$ because $\sigma$ was taken arbitrary. Since $\liem^\Q_\rho$ is a rational form, the restriction of $f$ to $\liem^\Q_\rho$ is invertible and $f$ induces an automorphism of $\liem^\Q_\rho$. 

For the other direction, fix $\sigma \in \Gal(E,\Q)$ and fix a basis $\{v_1, \ldots v_m\} \subseteq \liem^\Q_\rho$ for the Lie algebra $\lien^E$. Since $f$ induces an automorphism on $\liem^\Q_\rho$ we know that $f(v_i) \in \liem^\Q_\rho$ and thus that $\big(f(v_i)\big)^\sigma = \rho_\sigma(f(v_i))$. Since the vectors $\rho_{\sigma}(v_i)$ also form a basis for the Lie algebra $\lien^E$, it suffices to prove the relation $f^\sigma = \rho_\sigma f \rho_{\sigma^{-1}}$ on this basis. A computation shows that 
\begin{align*}
\rho_\sigma f \rho_{\sigma^{-1}} (\rho_\sigma(v_i)) = \rho_\sigma (f(v_i)) = \big(f(v_i)\big)^\sigma = f^\sigma(v_i^\sigma) = f^\sigma(\rho_\sigma(v_i)) 
\end{align*}
and thus the relation holds.
\end{proof}

The main theorem of this article is the combination of Proposition \ref{GC} and Lemma \ref{auto}: 
\begin{Thm}
\label{main}
Let $\lien^\Q$ be a rational Lie algebra and $\rho: \Gal(E,\Q) \to \Aut(\lien^\Q)$ a representation. Suppose there exists a Lie algebra automorphism $f: \lien^E \to \lien^E$ such that $\rho_\sigma f \rho_{\sigma^{-1}} = f^\sigma$ for all $\sigma \in \Gal(E,\Q)$. Then there also exists a rational form $\liem^\Q \subseteq \lien^E$ such that $f$ induces an automorphism of $\liem^\Q$. 

If all eigenvalues of $f$ are algebraic units of absolute value different from $1$, then $\liem^\Q$ is Anosov.
\end{Thm}

\begin{proof}
The only statement left to show is the last one. We claim that if $f$ is an automorphism of a rational Lie algebra with only algebraic units as eigenvalues, then $f$ is integer-like. Note that the coefficients of the characteristic polynomial of $f$ are formed by taking sums and products of the eigenvalues and thus all these coefficients are algebraic integers. Since these coefficients also lie in $\Q$, they are integers. The determinant of $f$ is equal to the product of all eigenvalues and therefore is an algebraic unit. Since the only algebraic units in $\Q$ are $\pm 1$, the claim now follows and this ends the proof.
\end{proof}

Note that the type of the Lie algebra $\liem^\Q$ of the theorem is equal to the type of $\lien^\Q$ and is thus completely determined. The signature of the automorphism $f$ also does not change by restricting it to a rational form. But in general the Lie algebras $\lien^\Q$ and $\liem^\Q$ will not be isomorphic, so this theorem does not allow us to show that a specific Lie algebra is Anosov. In Section \ref{verb} we determine the isomorphism class of the Lie algebra in a low-dimensional case by using the classification of low-dimensional Anosov algebras given in \cite{lw09-1}.

\smallskip

In the special case where the Lie algebra $\lien^\Q$ is given by the eigenspaces of an automorphism, the theorem becomes:

\begin{Cor}
\label{main2}
Let $\lien^\Q$ be a rational Lie algebra and assume there exists a decomposition of $\lien^\Q$ into subspaces $$\lien^\Q = \bigoplus_{\lambda \in E} V_\lambda$$ such that $[V_\lambda,V_\mu] \subseteq V_{\lambda \mu}$. Let $\rho:\Gal(E,\Q) \to \Aut(\lien^\Q)$ be a representation such that $\rho_\sigma(V_\lambda)= V_{\sigma(\lambda)}$ for all $\sigma \in \Gal(E,\Q)$. Then the linear map $f: \lien^E \to \lien^E$ given by $f(X) = \lambda X$ for all $X \in V_\lambda$ induces an automorphism on some rational form $\liem^\Q \subseteq \lien^E$. 

If every $\lambda$ is an algebraic unit of absolute value different from $1$, then $\liem^\Q$ is Anosov.
\end{Cor}

\begin{proof}
We will make use of Theorem \ref{main} to prove this. The condition on the Lie bracket implies that $f$ is a Lie algebra automorphism and the last condition is identical to the last condition of Theorem \ref{main}. So it is left to show that $$\rho_\sigma f \rho_{\sigma^{-1}} = f^\sigma$$ for all $\sigma \in \Gal(E,\Q)$. It suffices to prove this relation for vectors $X \in V_\lambda$ for all possible $\lambda$. For such a vector, we have that
\begin{align*}
\rho_\sigma f \rho_{\sigma^{-1}} \left( X \right) = \rho_\sigma \Big(f \big( \rho_{\sigma^{-1}}\left( X \right) \big) \Big)= \sigma^{-1}(\lambda) \rho_\sigma \big( \rho_{\sigma^{-1}} \left( X \right) \big) = \sigma^{-1}\left( \lambda \right) X
\end{align*}
since $\rho_{\sigma^{-1}}(X) \in V_{\sigma^{-1}(\lambda)}$ and
\begin{align*}
f^\sigma \left(X \right) =  \left( f(X) \right)^\sigma =  \left( \lambda X \right)^\sigma = \sigma^{-1}(\lambda)  X
\end{align*}
since $X^{\sigma^{-1}} = X$ and thus equality holds.
\end{proof}

\begin{Rmk}
\label{oned}
Corollary \ref{main2} is the version of the main theorem we will use most of the times. In many examples, the spaces $V_\lambda$ will be one-dimensional and thus given by basis vectors $X_\lambda$ for $\lambda \in E$. In these examples, the Lie bracket is of the form $$[X_\lambda,X_\mu] = \pm X_{\lambda\mu}$$ and the representation $\rho$ is given by $\rho_\sigma(X_\lambda) = \pm X_{\sigma(\lambda)}$ for all $\lambda \in E, \sigma \in \Gal(E,\Q)$. To use the corollary, the only thing left to check is that $\rho$ is indeed a representation, since all other conditions are straightforward.
\end{Rmk}

\section{Consequences of Theorem \ref{main}}
\label{conseq}

The main application of the theorem and corollary above lies in constructing Anosov Lie algebras of specific types and with Anosov automorphisms of specific signatures. We present here three different consequences of the main theorem, going from simplifying and correcting existing results to constructing new examples of minimal signature and minimal type.

\subsection{Simplification of existing results}

By using our main theorem, all the examples of Anosov Lie algebras in \cite{laur03-1,lw08-1,lw09-1,mw07-1} are now straightforward to construct. For instance, the examples in \cite{lw09-1,mw07-1} follow from Remark \ref{oned} since they start from a basis for the one-dimensional eigenspaces of a hyperbolic automorphism. As another example of the simplification we demonstrate how \cite[Theorem 3.2.]{lw08-1} follows after a few lines.

\smallskip

Recall that a Lie algebra $\lien^E$ over a field $E$ is called graded if there exist subspaces $\lien_i^E \subseteq \lien^E$ for $i \in \{1, \ldots, k\}$ such that 
\begin{align*}
&\lien^E = \lien_1^E \oplus \lien_2^E \oplus \ldots \oplus \lien_k^E \hspace{1mm} \text{  and} \\
&[\lien_i^E, \lien_j^E] \subseteq \lien_{i+j}^E.
\end{align*}
If $\lien^E$ is graded, then there exists an automorphism $f_\lambda: \lien^E \to \lien^E$ for every $\lambda \in E, \lambda \neq 0$, which is defined by $$f_\lambda (X) = \lambda^i X  \hspace{2mm} \forall X \in \lien_i.$$
The proof of \cite[Theorem 3.2.]{lw08-1} now follows immediately from the main theorem:

\begin{Thm}
\label{laur}
Let $\lien^\Q$ be a graded Lie algebra and consider the direct sum $$\tilde{\lien}^\Q= \underbrace{\lien^\Q \oplus \ldots \oplus \lien^\Q}_{m\ {\rm times}}$$ with $m \geq 2$. Then there exists a rational form of $\tilde{\lien}^\R$ which is Anosov.
\end{Thm}

\begin{proof}
Let $E$ be any real Galois extension of $\Q$ of degree $m$ with $\Gal(E,\Q) = \{ \sigma_1, \ldots, \sigma_m\}$. Every $\sigma \in \Gal(E,\Q)$ induces a permutation $\pi \in S_m$ via $$\sigma \sigma_i = \sigma_{\pi(i)}  \hspace{1mm}  \text{ for all } \hspace{1mm} i \in \{1, \ldots, m\},$$ and thus also an automorphism $\rho_\sigma \in \Aut(\tilde{\lien}^\Q)$ by permuting the components of $\tilde{\lien}^\Q$ according to $\pi$. Note that $\rho$ is a representation $\Gal(E,\Q) \to \Aut(\tilde{\lien}^\Q)$. 

\smallskip

Take an unit Pisot number $\lambda$ in $E$, which is possible since $m \geq 2$ (see Appendix for more details) and take the grading $ \lien^\Q = \lien_1^\Q \oplus \dots \oplus \lien_k^\Q$ for $\lien^\Q$. By writing the subspace $\lien^\Q_i$ of the $j$-th component of $\tilde{\lien}^\Q$ as $V_{\sigma_i(\lambda^j)}$, it is immediate from the construction of $\rho$ that all conditions of Corollary \ref{main2} are satisfied. Since every $\sigma_i(\lambda^j)$ is an algebraic unit of absolute value different from $1$, it follows that $\tilde{\lien}^E$ (and therefore also $\tilde{\lien}^\R$) has a rational form which is Anosov.
\end{proof}

\subsection{Correction to a result of \cite{lw09-1}}

\label{verb}

In \cite{lw09-1}, a classification of all Anosov Lie algebras up to dimension $8$ is given. As one of the consequences, as we recalled in (False) Claim \ref{corlaur}, it is stated in \cite[Corollary 4.3.]{lw09-1} that every Anosov diffeomorphism on a nilmanifold of dimension $\leq 8$ which is not a torus, has signature $\{3,3\}$ or $\{4,4\}$. There is not really a proof of this statement given in \cite{lw09-1} and in fact, by using Theorem \ref{main} we can give new examples which were overlooked by the author.

\smallskip

First we recall some notions of \cite{lw09-1} about the Pfaffian form of a Lie algebra. Let $\lien^E$ be a Lie algebra over any field $E$ and take $\langle \hspace{1mm},\hspace{1mm} \rangle$ an inner product on $\lien^E$, i.e.\ a non-degenerate symmetric bilinear form. For every $Z \in \lien^E$, there exists a linear map $J_Z: \lien^E \to \lien^E$ defined by $$\langle J_Z X, Y \rangle = \langle [X,Y], Z \rangle \hspace{3mm} \forall X,Y \in \lien^E.$$ Note that $J_Z$ is skew-symmetric with respect to the inner product, meaning that $$\langle J_Z X, Y \rangle = - \langle X, J_Z Y \rangle$$ for all $X,Y \in \lien^E$. It is easy to check that an isomorphism $\alpha: \lien^E \to \lien^E$ is an automorphism of $\lien^E$ if and only if \begin{align}
\label{aut} \alpha^T J_Z \alpha = J_{\alpha^T (Z)}
\end{align}
 for all $Z \in \lien^E$, where $\alpha^T: \lien^E \to \lien^E$ is the adjoint map of $\alpha$, defined by $\langle X, \alpha^T( Y) \rangle = \langle \alpha (X), Y \rangle$.

Now assume that $\lien^E$ is a $2$-step nilpotent Lie algebra of type $(2 m,k)$ and take $V^E \subseteq \lien^E$ such that $V^E \oplus \gamma_2(\lien^E) = \lien^E$ as a vector space. Take an inner product satisfying $\langle V^E , \gamma_2(\lien^E)\rangle = 0$ and such that there exists an orthonormal basis for $V^E$ with respect to the inner product. 
Denote the vector space (in fact it is a Lie algebra) of all skew-symmetric endomorphisms of $V^E$ by $\lies(V^E)$. The construction above then induces a linear map \begin{align*}J : \gamma_2(\lien^E) &\to \lies(V^E)\\ Z &\mapsto J_Z |_{V^E}.\end{align*}
After taking an orthonormal basis for $V^E$, we can identify $\lies(V^E)$ with the skew-symmetric matrices. Recall that the Pfaffian on $V^E$ is the unique polynomial function $$\Pf: \lies(V^E) \to E$$ such that $\Pf(A)^2 = \det(A)$ for all $A \in \lies(V^E)$ and $\Pf(S) = 1$ for some fixed $S \in \lies(V^E)$ with $\det(S) = 1$ (where this last condition is needed to fix the sign). The Pfaffian obviously satisfies the relation $\Pf(A^T B A) = \det(A) \Pf(B)$ for all endomorphisms $A: V^E \to V^E$. The composition $$h = \Pf \circ J: \gamma_2(\lien^E) \to E$$ is called the Pfaffian form of the 2-step nilpotent Lie algebra $\lien^E$. By taking another vector space $V^E$ or changing the inner product on $\lien^E$ or the basis of $V^E$, the Pfaffian form changes to a polynomial \begin{align*}\gamma_2(\lien^E) &\to E \\ Z &\mapsto e h(\beta(Z))\end{align*} for some $e \in E \setminus \{ 0 \}$ and some isomorphism $\beta: \gamma_2(\lien^E) \to \gamma_2(\lien^E)$. Thus the Pfaffian form is uniquely determined by $\lien^E$ up to projective equivalence (see \cite[Proposition 2.4.]{laur08-1} for the exact definition and proof of this statement). The polynomial $h$ is homogeneous of degree $m$ in $k$ variables and with coefficients in $E$.  An automorphism of the Pfaffian form is an isomorphism $\beta: \gamma_2(\lien^E) \to \gamma_2(\lien^E)$ such that $h\circ \beta =h$. 

A binary quadratic form over $\Q$ is a homogeneous polynomial of degree 2 in 2 variables, i.e.\ polynomials that can be written as $$h(X,Y) = a X^2 + b XY + c Y^2$$ with $a,b,c \in \Q$. The Pfaffian form of a rational Lie algebra of type $(4,2)$ is such a polynomial. The discriminant $\Delta(h)$ of $h$ is defined as $$\Delta(h) = b^2 - 4 ac.$$ In \cite{gss82-1} it is shown that two rational Lie algebras of type $(4,2)$ with Pfaffian forms $h_1$ and $h_2$ are isomorphic if and only if there exists $q \in \Q \setminus \{ 0 \}$ such that $\Delta(h_1) = q^2 \Delta(h_2)$. Given any $k \in \Z$, we can consider the Lie algebra $\lien^\Q_k$ given by the basis $X_1, X_2, X_3, X_4, Z_1, Z_2$ and relations \begin{equation*}
\begin{aligned}[c]
[X_1,X_3] &= Z_1  \\ 
[X_2,X_3] &= k Z_2 
\end{aligned}
\hspace{2cm}
\begin{aligned}[c]
 [X_1,X_4] &= Z_2 \\ 
[X_2,X_4] &= Z_1.
\end{aligned}
\end{equation*}
The Lie algebra $\lien^\Q_k$ has Pfaffian form equal to $h(X,Y) = X^2 - k Y^2$ with discriminant $4k$. So every rational Lie algebra of type $(4,2)$ is isomorphic to $\lien^\Q_k$ for some $k \in \Z$ and $\lien^\Q_k$ is isomorphic to $\lien^\Q_{k^\prime}$ if and only if there exists a natural number $q > 0$ such that $k = q^2 k^\prime$.

The automorphisms $\beta \in \SL(2,\Z)$ of a binary quadratic form $h(X,Y) = a X^2 + b XY + c Y^2$ are completely known and described e.g.\ in \cite[Theorem 2.5.5.]{bv07-1}. Given a solution $x,y \in \Z$ of the Pell equation $$x^2 - \Delta(h) y^2 = 4,$$ the matrix $$U(x,y) = \begin{pmatrix} \frac{ x - yb}{2} & -cy \\ ay & \frac{x + yb}{2} \end{pmatrix} \in \SL(2,\Z)$$ is an automorphism of the quadratic form $h$. The map $U$ is a bijection between the solutions of the Pell equation and the automorphisms of $h$ which lie in $\SL(2,\Z)$. The eigenvalues of $U(x,y)$ are equal to $\frac{x \pm \sqrt{\Delta(h)} y}{2}$, so the field in which these eigenvalues lie gives us information about the discriminant of the form $h$.

Let $\lien^\Q$ be a rational Lie algebra of type $(4,2)$ and $\alpha \in \Aut(\lien^\Q)$ an Anosov automorphism. By squaring $\alpha$ if necessary, we can also assume that $\det(\alpha) = 1$. From the equation $\alpha^T J_Z \alpha = J_{\alpha^T (Z)}$ (see (\ref{aut})) and by applying the Pfaffian, we get that $h(\alpha^T (Z)) = h(Z)$ for all $Z \in V^\Q$. So $\alpha$ induces a hyperbolic and integer-like automorphism $\beta = \alpha^T |_{\gamma_2(\lien^\Q)}$ of the Pfaffian form $h$. The eigenvalues of $\beta$ (and thus also of $\alpha |_{\gamma_2(\lien^\Q)}$) lie in the field $\Q(\sqrt{\Delta(h)})$ and thus if we know these eigenvalues, we can determine the discriminant of the Pfaffian form of $\lien^\Q$ up to a square and therefore also the isomorphism class of $\lien^\Q$. This also implies that every Anosov Lie algebra of type $(4,2)$ is isomorphic to a Lie algebra $\lien^\Q_k$ with $k$ a square free natural number $>1$ (since all other values of $k$ imply that the eigenvalues have absolute value $1$).

\smallskip

We now have all the tools to construct new Anosov automorphisms on the Anosov Lie algebras $\lien^\Q_k$:

\begin{Prop}
\label{count}
Let $k$ be a natural number with $k > 1$ and $k$ square free. Let $\lien^\Q_k$ be the Lie algebra with basis $X_1, X_2, X_3, X_4, Z_1, Z_2$ and relations 
\begin{equation*}
\begin{aligned}[c]
[X_1,X_3] &= Z_1  \\ 
[X_2,X_3] &= k Z_2 
\end{aligned}
\hspace{2cm}
\begin{aligned}[c]
 [X_1,X_4] &= Z_2 \\ 
[X_2,X_4] &= Z_1.
\end{aligned}
\end{equation*}
Then there exists an Anosov automorphism $f$ on $\lien^\Q_k$ with $\sg(f) = \{2,4\}$.
\end{Prop}

\begin{proof}
Fix the $k$ of the theorem and let $l$ be a different natural number with $l$ square free and $l>1$. Take $E = \Q(\sqrt{k},\sqrt{l})$, then $E$ is Galois over $\Q$ with $\Gal(E,\Q) = \Z_2 \oplus \Z_2$. Let $\tau \in \Gal(E,\Q)$ be the unique element with $\tau(\sqrt{k}) = \sqrt{k}, \tau(\sqrt{l}) = -\sqrt{l}$ and take another $\sigma \in \Gal(E,\Q)$ such that $\Gal(E,\Q) = \{ 1, \sigma, \tau, \sigma \tau\}$. Take $\lambda_1 \in E$ an unit Pisot number as introduced in the Appendix. Write the Galois conjugates of $\lambda_1$ as $\lambda_1, \tau(\lambda_1) = \lambda_2, \sigma(\lambda_1) = \lambda_3, \sigma \tau(\lambda_1) = \lambda_4$ 

Consider the Lie algebra $\lien^\Q$ with basis $X_{\lambda_1},X_{\lambda_2},X_{\lambda_3}, X_{\lambda_4}, Y_{\lambda_1 \lambda_2}, Y_{\lambda_3 \lambda_4}$ and Lie bracket given by 
\begin{align*}
[X_{\lambda_1},X_{\lambda_2}] &= Y_{\lambda_1 \lambda_2} \\ 
[X_{\lambda_3}, X_{\lambda_4}] &= Y_{\lambda_3 \lambda_4}
\end{align*} 
and all other brackets zero (so $\lien^\Q$ is isomorphic to the direct sum of two copies of the Heisenberg algebra of dimension $3$). Each of these basis vectors spans a $1$-dimensional subspace indexed by the same algebraic unit, corresponding to the decomposition in Corollary \ref{main2}. Consider the representation $\rho: \Gal(E,\Q) \to \Aut(\lien^\Q)$ induced by $\rho_\sigma(X_\lambda) = X_{\sigma(\lambda)}$ and $\rho_\tau(X_\lambda) = X_{\tau(\lambda)}$. A small computation shows that this is indeed a representation. By using Corollary \ref{main2} we then get a rational form $\liem^\Q$ of $\lien^E$ with Anosov automorphism $f: \liem^\Q \to \liem^\Q$. 

Note that $f$ has only two eigenvalues $>1$, namely $\lambda_1$ and $\lambda_1\lambda_2$ and thus $\sg(f) = \{2,4\}$. Since $\tau(\lambda_1 \lambda_2) = \lambda_1 \lambda
_2$ and $\Q(\sqrt{k})$ is the unique subfield of $E$ fixed by $\tau$, the eigenvalue $\lambda_1 \lambda_2$ is in $\Q(\sqrt{k})$, showing that the Pfaffian form $h$ of $\liem^\Q$ satisfies $\Delta(h) = q^2 k$ for some $q \in \Q$. This shows that $\liem^\Q$ is isomorphic to $\lien^\Q_k$ and thus $\lien^\Q_k$ also has an Anosov automorphism of signature $\{2,4\}$.
\end{proof}

It is also possible to start the proof from a Galois extension $\Q \subseteq E$ with $\Gal(E,\Q) = \Z_4$ as we show in the example below. By computing a basis for $\liem^\Q$ we give an explicit example which was overlooked in (False) Claim \ref{corlaur}:
\begin{Ex} Start from the polynomial $$p(X) = X^4 -4X^3 - 4X^2+X+1,$$ which has $4$ distinct real roots, say $\lambda_1 > \lambda_2 > \lambda_3 > \lambda_4$. These roots satisfy $\lambda_1 > 1$ and $\vert \lambda_i \vert < 1$ for $i \in \{ 2, 3, 4 \}$, showing that $\lambda_1$ is an unit Pisot number. The Galois group of the field $E = \Q(\lambda_1)$ is isomorphic to $\Z_4$, which can be checked e.g.\ with GAP, see \cite{gap14-1}. A generator $\sigma \in \Gal(E,\Q)$ is given by $$\sigma(\lambda_1) = \lambda_2, \hspace{1mm} \sigma(\lambda_2) = \lambda_3, \hspace{1mm} \sigma(\lambda_3) = \lambda_4, \hspace{1mm} \sigma(\lambda_4) = \lambda_1.$$ Just as in the proof of Proposition \ref{count}, consider the Lie algebra $\lien^\Q$ with basis $$X_{\lambda_1},X_{\lambda_2},X_{\lambda_3}, X_{\lambda_4}, Y_{\lambda_1 \lambda_3}, Y_{\lambda_2 \lambda_4}$$ and Lie bracket given by 
\begin{align*}
[X_{\lambda_1},X_{\lambda_3}] &= Y_{\lambda_1 \lambda_3} \\ 
[X_{\lambda_2}, X_{\lambda_4}] &= Y_{\lambda_2 \lambda_4}
\end{align*} 
Consider the representation $\rho: \Gal(E,\Q) \to \Aut(\lien^\Q)$ induced by $\rho_\sigma(X_{\lambda_i}) = X_{\sigma(\lambda_i)}$. The main theorem guarantees us the existence of a rational form $\liem^\Q$ of $\lien^E$ which is Anosov, but we now compute this Lie algebra explicitly by giving a basis for $\liem^\Q$.

\smallskip

Consider the basis $U_1, U_2, U_3, U_4, V_1, V_2$ given by 
\begin{eqnarray*}
U_i  &=& \sum_{j=1}^4 \lambda_j ^{i-1} X_{\lambda_j } \\
V_i  &=& \left( \lambda_3^{i} - \lambda_1^{i} \right) Y_{\lambda_1 \lambda_3} + \left(\lambda_4^{i} - \lambda_2^{i}\right) Y_{\lambda_2 \lambda_4}.
\end{eqnarray*}
To simplify the computations, we will use the notations $\lambda_0 = \lambda_4$ and $\lambda_5=\lambda_1$. The basis vectors $U_i$ satisfy $$U_i^\sigma = \sum_{j=1}^4 \left(\lambda_j^{i-1}\right)^\sigma X_{\lambda_{j}} = \sum_{j=1}^4 \sigma^{-1}\left(\lambda_j^{i-1}\right) X_{\lambda_{j}} = \sum_{j=1}^4 \lambda_{j-1}^{i-1} X_{\lambda_{j}}$$ and $$ \rho_{\sigma}(U_i) = \sum_{j=1}^4 \lambda_j ^{i-1} \rho_\sigma\left(X_{\lambda_j }\right) = \sum_{j=1}^4 \lambda_j^{i-1} X_{\lambda_{j+1}} = \sum_{j=2}^5 \lambda_{j-1} ^{i-1} X_{\lambda_{j}}.$$ We conclude that $\rho_{\sigma}(U_i) = U_i^\sigma$ and similarly this equation also holds for the vectors $V_i$. This shows that the basis vectors $U_1, U_2, U_3, U_4, V_1, V_2$ satisfy the defining relation of the rational form given in equation (\ref{def}) of Section \ref{secGC} and thus they indeed span the rational form $\liem^\Q$ of Theorem \ref{main}. The induced Anosov automorphism on $\liem^\Q$ guaranteed by Theorem \ref{main} is given by the matrix
\begin{eqnarray*}
\begin{pmatrix}
0 & 0 & 0 & -1 & 0 & 0\\
1 & 0 & 0 & -1 & 0 & 0\\
0 & 1 & 0 & 4 & 0 & 0\\
0 & 0 & 1 & 4 & 0 & 0\\
0 & 0 & 0 & 0 & -\frac{1}{2} & -\frac{1}{2} \\[0.3em]
0 & 0 & 0 & 0 & -\frac{1}{2} & -\frac{5}{2}
\end{pmatrix}
\end{eqnarray*}
in the basis $\hspace{0.5 mm} U_1, U_2, U_3, U_4, V_1, V_2$. By using the matrix representation of this Anosov automorphism, one can compute the Lie bracket:
\begin{equation*}
\begin{aligned}[c]
\left[U_1, U_2 \right] &= V_1 \\ 
\left[U_2, U_3 \right] &= -\frac{1}{2} V_1 - \frac{1}{2} V_2 \\ 
\left[U_3, U_4 \right] &= \frac{1}{2} V_1 + \frac{3}{2} V_2
\end{aligned}
\hspace{2cm}
\begin{aligned}[c]
\left[U_1, U_3 \right] &= V_2 \\ 
\left[U_2, U_4 \right] &= -\frac{1}{2} V_1 - \frac{5}{2} V_2 \\
\left[U_1, U_4 \right] &= \frac{3}{2} V_1 + \frac{9}{2} V_2.
\end{aligned}
\end{equation*}
The discriminant of the Pfaffian form of $\liem^\Q$ is $\frac{5}{4}$, so $\liem^\Q$ is isomorphic to $\lien_5^\Q$. The characteristic polynomial of the Anosov automorphism restricted to $\gamma_2(\liem^\Q)$ is equal to $X^2+3X+1$ which has $\Q(\sqrt{5})$ as splitting field.
\end{Ex}

Of course, Proposition \ref{count} also gives us examples of Anosov automorphisms $f: \lien^\Q_k \oplus \Q^2 \to \lien^\Q_k \oplus \Q^2$ with $\sg(f) = \{3,5\}$. But by using the notion of Scheuneman duality (see \cite{sche67-1}) for $2$-step nilpotent Lie algebras, it is possible to give another class of Anosov automorphisms overlooked in (False) Claim \ref{corlaur}. First we recall some details about this method as described in \cite{laur08-1}.

\smallskip

Let $V^\Q$ be any vector space with an inner product and consider the standard inner product $B$ on $\lies(V^\Q)$, given by $$B(Z_1,Z_2) = \trace(Z_1^T Z_2) = - \trace(Z_1 Z_2).$$ For every subspace $W^\Q$ of $\lies(V^\Q)$, there exists a rational Lie algebra $\lien^\Q = V^\Q \oplus W^\Q$ with Lie bracket $[\hspace{1mm},\hspace{1mm}]: V^\Q \times V^\Q \to W^\Q$ defined by $$B([X,Y], Z) = \langle Z(X),Y \rangle$$ for all $Z \in \lies(V^\Q)$, $X,Y \in V^\Q$. This is a $2$-step nilpotent Lie algebra where the map $J: \gamma_2(V^\Q \oplus W^\Q) = W^\Q \to \lies(V^\Q)$ is the inclusion. If we take the vector space $W^\Q = \lies(V^\Q)$, then the result is the free $2$-step nilpotent Lie algebra on $V^\Q$. An isomorphism $\alpha: V^\Q \to V^\Q$ induces an automorphism $\bar{\alpha}: \lien^\Q \to \lien^\Q$ if and only if $\alpha^T Z \alpha \in W^\Q$ for all $Z \in W^\Q$.

Let $\lien^\Q$ be a Lie algebra of type $(m,k)$ and consider the map $J: \gamma_2(\lien^\Q) \to \lies(V^\Q)$ as introduced above. Denote the image of $J$ as $W^\Q$, then it follows by definition that $\lien^\Q$ is isomorphic to the Lie algebra $V^\Q \oplus W^\Q$ of the previous paragraph. The dual of $\lien^\Q$ is then the Lie algebra $\tilde{n}^\Q = V^\Q \oplus \tilde{W}^\Q$ with Lie bracket as in the previous paragraph, where $\tilde{W}^\Q$ is the orthogonal complement of $W^\Q$ in $\lies(V^\Q)$ relative to the inner product $B$ given above. The dual of the Lie algebra $\lien^\Q_k$ is denoted by $\lieh_k^\Q$. If $\alpha \in \GL(V^\Q)$ induces an automorphism $\bar{\alpha}$ of $\lien^\Q$, then $\alpha^T$ induces an automorphism on $\tilde{\lien}^\Q$ since $\alpha \tilde{W}^\Q \alpha^T = \tilde{W}^\Q$ and this map is called the dual automorphism of $\bar{\alpha}$. The combined eigenvalues of $\bar{\alpha}$ and its dual on $\gamma_2(\lien^\Q)$ and $\gamma_2(\tilde{\lien}^\Q)$ are equal to the eigenvalues of the map that $\alpha$ induces on $\gamma_2(V^\Q \oplus \lies(V^\Q))$ where $V^\Q \oplus \lies(V^\Q)$ is the free $2$-step nilpotent Lie algebra on $V^\Q$. 

\smallskip

The dual Lie algebra of $\lien^\Q_k$ is of type $(4,4)$ and denoted as $\lieh^\Q_k$. The Lie algebra $\lieh^\Q_k$ can also be described as the one with basis $X_1, X_2, X_3, X_4, Z_1, Z_2, Z_3, Z_4$ and relations 
\begin{equation*}
\begin{aligned}[c]
[X_1,X_2] &= Z_1  \\ 
[X_1,X_3] &= Z_2 \\
[X_1,X_4] &= kZ_3
\end{aligned}
\hspace{2cm}
\begin{aligned}[c]
[X_2,X_3] &= -Z_3  \\ 
[X_2,X_4] &= -Z_2 \\
[X_3,X_4] &= Z_4,
\end{aligned}
\end{equation*}
see for example \cite{laur08-1}. From the Scheuneman duality, the following proposition is immediate:

\begin{Prop}
For every $k \in \N$ with $k > 1$ and $k$ not a square, there exists an Anosov automorphism on $\lieh^\Q_k$ with $\sg(f) = \{3,5\}$. On $\lieh^\Q_1$ every Anosov automorphism has signature $\{4,4\}$.
\end{Prop}

\begin{proof}
Note that the dual of an Anosov automorphism with $\sg(f)= \{2,4\}$ is also Anosov with signature $\{3,5\}$ and vice versa. So the first part follows from the fact that $\lieh^\Q_k$ is the dual of $\lien^\Q_k$. Also the second statement follows since the Lie algebra $\lien^\Q_1$ is not Anosov.
\end{proof}

For all other Lie algebras, (False) Claim \ref{corlaur} is correct (and the arguments to prove it are the same as the ones used to prove the classification of Anosov Lie algebras up to dimension $8$). Also, the Lie algebra $\lieh^\Q_k$ does not admit an Anosov automorphism of signature $(2,6)$, for example by using the same number theoretical arguments as in \cite{lw09-1}. So the combined results above determine completely for which Anosov Lie algebras (False) Claim \ref{corlaur} is indeed false.

 Note that the examples of Proposition \ref{count} also answer Question \ref{q1} about non-abelian examples of signature $\{2,q\}$ for some $q \in \N_0$. We give a more general approach to this question in the next section.

\subsection{Anosov automorphisms of minimal signature and minimal type}
\label{sign}

In this subsection, we show how the main theorem can be used to construct Anosov automorphisms of minimal signature and Anosov Lie algebras of minimal type. These examples answer Questions \ref{q1} and \ref{q2} which we already mentioned in the introduction.

\smallskip

First we recall some basic properties of unit Pisot numbers. If $E$ is a real Galois extension of $\Q$ of degree $n$, then we call an algebraic integer a Pisot number if $ \lambda  > 1$ and for all $1 \neq \sigma \in \Gal(E,\Q)$, it holds that $\vert \sigma(\lambda) \vert  <1$. An unit Pisot number is then an algebraic unit which is also a Pisot number. We say that an algebraic unit $\lambda$ with Galois conjugates $\lambda_1= \lambda, \ldots, \lambda_n$ satisfies the full rank condition if for all $d_1, \ldots, d_n$ integers with $$\prod_{j=1}^n \lambda_j^{d_j} = \pm 1,$$ it must hold that $d_1 = d_2 = \ldots = d_n$. From \cite[Proposition 3.6.]{payn09-1} it follows that every unit Pisot number satisfies the full rank condition.

\smallskip

Let $\lien^\Q$ be an Anosov Lie algebra of nilpotency class $c$ and $f: \lien^\Q \to \lien^\Q$ a hyperbolic integer-like automorphism with signature $\{p,q\}$. The characteristic polynomial $h(X)$ of $f$ has integer coefficients and constant term $\pm 1$. This implies that if $g(X)$ is a rational polynomial which divides $h(X)$, then it must have at least one root of absolute value strictly smaller than $1$. We know that $f$ induces an isomorphism on each quotient $\faktor{\gamma_{i-1}(\lien^\Q)}{ \gamma_i(\lien^\Q)}$ and thus the polynomial $h(X)$ has at least $c$ irreducible factors. Therefore $f$ has at least $c$ eigenvalues of absolute value strictly smaller than $1$ and thus $p \geq c$. By considering $f^{-1}$ as well, we get that $q \geq c$ and this shows that $\min (\sg(f)) \geq c$. We say that an Anosov automorphism $f: \lien^\Q \to \lien^\Q$ has minimal signature if equality holds, i.e.\ if $\min(\sg(f)) = c$. 

\smallskip

Question \ref{q1} asks if there exists Anosov automorphisms of minimal signature for $c= 2$ and already in Section \ref{verb} we gave a positive answer to this question as a consequence of the main theorem. So the existence of Anosov automorphisms of minimal signature is a generalization of Question \ref{q1} and with Theorem \ref{main} we can also give a positive answer to the generalized question:

\begin{Thm}
\label{csig}
For every $c$, there exists an Anosov automorphism $f: \lien^\Q \to \lien^\Q$ on a Lie algebra of nilpotency class $c$ such that $f$ is of minimal signature. 
\end{Thm}

\begin{proof}
Let $E$ be a real Galois extension of $\Q$ with $\Gal(E,\Q) \cong \Z_{2n}$ for some $n > 1$ and $\sigma$ a generator of $\Gal(E,\Q)$. Take $\lambda$ an unit Pisot number in $E$ with the extra condition that $\vert \lambda \sigma^n(\lambda^2) \vert < 1$ (see the Appendix for more details). Since $n > 1$, we have that $ \vert \lambda \sigma^n(\lambda) \vert > \vert \prod_{i=1}^{2n} \sigma^i(\lambda) \vert = 1$. Consider the collection of algebraic integers $$\mu_{i,j} = \sigma^i(\lambda^{j-1}) \sigma^{i+n}(\lambda)$$ for all $i \in \{1, \ldots, 2n\}$ and all $j \in \{1,\ldots, c\}$. The Galois conjugates $\sigma^i(\lambda)$ of $\lambda$ are the $\mu_{i,j}$ with $j = 1$. Note that the definition implies that $\mu_{i+n,2} = \mu_{i,2}$ and all other $\mu_{i,j}$ are distinct because of the full rank condition. 

Every $\mu_{i,j}$ with $i \notin \{n,2n\}$ has absolute value $<1$ since $\lambda$ is a Pisot number. The algebraic unit $\mu_{n,3} = \lambda \sigma^n(\lambda^2)$ satisfies $\vert \mu_{n,3} \vert < 1$ because of our choice of $\lambda$ and therefore also all $\mu_{n,j} = \sigma^n\left(\lambda^{j-3}\right) \mu_{n,3} $ with $j \geq 3$ have absolute value $<1$. Thus it follows that of all $\mu_{i,j}$, only $$\mu_{n,1}= \lambda,\mu_{n,2} = \mu_{2n,2} =  \lambda \sigma^n(\lambda), \mu_{2n,3}= \lambda^2 \sigma^n(\lambda), \ldots, \mu_{2n,c} = \lambda^{c-1} \sigma^n(\lambda)$$ have absolute value $>1$, so in total there are $c$ of the $\mu_{i,j}$ with $\vert \mu_{i,j} \vert > 1$.

Now consider the Lie algebra $\lien^\Q$ with basis $X_{\mu_{ij}}$ for all values of $i$ and $j$, where we write the $\mu_{i,1}$ as the conjugates of $\lambda$. The Lie bracket on $\lien^\Q$ is given by on the one hand 
\begin{align*}
[ X_{\sigma^i(\lambda)}, X_{\sigma^{i+n}(\lambda)}] = X_{\mu_{i,2}}
\end{align*}
for all $i \in \{1, \ldots, n\}$ and on the other hand by 
\begin{align*}
[ X_{\sigma^i(\lambda)}, X_{\mu_{i,j}}] = X_{\mu_{i,j+1}}
\end{align*}
for all $i \in \{1,\ldots, 2n\}, \hspace{1mm} j \in \{2, \ldots, c\}$ (and all other brackets are $0$). It is easy to check that these relations define a Lie algebra (i.e.\ that the Jacobi identity holds) and that the $1$-dimensional subspaces spanned by each basis vector satisfy the conditions of Corollary \ref{main2}. The map $\rho_\sigma$ given by $\rho_\sigma(X_{\mu_{i,j}}) = - X_{\sigma(\mu_{i,j})}$ for $i \in \{ n, 2n \}$ and $j \geq 2$ and $\rho_\sigma(X_{\mu_{i,j}}) = X_{\sigma(\mu_{i,j})}$ for all other $(i,j)$ defines a representation $\rho: \Gal(E,\Q) \to \Aut(\lien^\Q)$. The minus sign in the first case comes from the relation $\mu_{i+n,2} = \mu_{i,2}$. The conditions of Corollary \ref{main2} are satisfied for $\rho$ and thus this gives us a rational form $\liem^\Q$ with Anosov automorphism $f$. There are only $c$ eigenvalues of $f$ with absolute value $>1$, so $f$ is of minimal signature.
\end{proof}

The type of the example constructed in this theorem is equal to $$\underbrace{(2n,n,2n, 2n,\ldots, 2n)}_{c\ {\rm components}}$$ for all $n \geq 2$. This is not the only possibility for Anosov automorphisms of minimal signature since one can construct examples on Lie algebras of type $$\underbrace{(2n,n,2n, n, 2n,\ldots)}_{c\ {\rm components}},$$ where the induced eigenvalues on $\faktor{\gamma_{2j}(\lien^\Q)}{ \gamma_{2j+1}(\lien^\Q)}$ are of the form $\left( \sigma^{i}\left(\lambda\right) \sigma^{n+i}\left(\lambda\right)\right)^j$. The construction of such examples is similar as in Theorem \ref{csig}. We conjecture that these are the only possibilities:

\begin{Con}
Let $f: \lien^\Q \to \lien^\Q$ be an Anosov automorphism of minimal signature, then the type of $\lien^\Q$ is one of the following:
\begin{enumerate}[(i)]
\item $(2n,n,2n, 2n,\ldots, 2n)$ or
\item $(2n,n,2n, n, 2n,\ldots)$,
\end{enumerate}
where $n > 1$.
\end{Con}

\noindent The methods of this paper are useful to prove or disprove this conjecture.

\medskip

In a similar way, the main theorem also gives a positive answer to Question \ref{q2}. We state this theorem in a more general setting:

\begin{Thm}
\label{last}
For every $c \in \N_0$ and $n \in \N$ with $n > 2$, there exists an Anosov Lie algebra of type $(n,\ldots,n)$ and nilpotency class $c$.
\end{Thm} 

\begin{proof}
Let $E \supseteq \Q$ be a real Galois extension with cyclic Galois group of order $n$ and let $\sigma \in \Gal(E,\Q)$ be a generator. Take $\lambda_1 \in E$ an unit Pisot number and consider the Galois conjugates $\lambda_1,\lambda_2 = \sigma(\lambda_1),\ldots, \lambda_n = \sigma^{n-1}(\lambda_1)$. Define the algebraic units $$\mu_{i,j} = \lambda_i^{j-1} \sigma(\lambda_i)$$ for all $i \in \{1,\ldots, n\}$ and $j \in \{1,\ldots, c\}$, where the algebraic units $\lambda_i$ occur as the $\mu_{i,j}$ with $j=1$. Let $\lien^\Q$ be the Lie algebra with basis $X_{\mu_{i,j}}$ for all $i\in \{1,\ldots,n\}$ and $j \in \{1,\ldots, c\}$, with Lie bracket given by 
\begin{align*}
[X_{\lambda_i},X_{\mu_{i,j}}] &= X_{\mu_{i,j+1}}
\end{align*}
for all $i\in \{1,\ldots,n\}$ and $j \in \{1, \ldots, c\}$ and all other brackets $0$. It is easy to see that the Jacobi identity holds and thus that $\lien^\Q$ is indeed a Lie algebra.

The linear map $h: \lien^\Q \to \lien^\Q$ defined by $h(X_{\mu_{i,j}}) = X_{\sigma(\mu_{i,j})}$ is an automorphism of this Lie algebra of order $n$. So the map $\rho: \sigma \mapsto h$ defines a representation $\rho: \Gal(E,\Q) \to \Aut(\lien^\Q)$. This Lie algebra and the representation $\rho$ satisfy the conditions of Corollary \ref{main2} (where the spaces $V_\lambda$ are one-dimensional and spanned by the basis vectors) and thus there exists a rational form of $\lien^E$ which is Anosov. The type of this rational form is equal to the type of $\lien^\Q$.
\end{proof}

The case where $n = 3$ gives an answer to Question \ref{q2}. It is an open question to determine all possibilities for the types $(n_1,\ldots, n_c)$ of Anosov Lie algebras with $n_1 = n_2 = 3$. To solve this problem, a careful study of the conjugates of algebraic units of degree $3$ is needed.

\section*{Appendix}
\label{app}

An important ingredient we used during this article is the existence of unit Pisot numbers in a real Galois extension $E$ of $\Q$. In this appendix we extend some results of \cite{dd13-1} about $c$-hyperbolic units to unit Pisot numbers with extra conditions on them.

\smallskip

If $E$ is a real Galois extension of $\Q$ of degree $n$, then we call an algebraic integer a Pisot number if $ \lambda  > 1$ and for all $1 \neq \sigma \in \Gal(E,\Q)$, it holds that $\vert \sigma(\lambda) \vert  <1$. An unit Pisot number is then an algebraic unit which is also a Pisot number. Denote by $U_E$ the algebraic units of $E$ and by $\sigma_1, \ldots, \sigma_n$ all elements of the Galois group $\Gal(E,\Q)$ with $\sigma_1 = 1$.  From Dirichlet's Unit Theorem we know that the map
\begin{align*}
l: U_E \to \R^{n}: \lambda \mapsto \Big( \log  \big( \vert \sigma_1\left(\lambda\right)\vert\big) , \ldots, \log\big( \vert \sigma_n \left(\lambda \right) \vert \big) \Big)
\end{align*}
maps $U_E$ onto a cocompact lattice of the subspace $V \subseteq \R^n$, where $V$ is given by the equation $x_1 + \ldots + x_n = 0$. The unit Pisot numbers are mapped to the open subset $O \subseteq V$ given by the equations $x_1 > 0$ and $x_i < 0$ for all $i \geq 2$. So for the existence of unit Pisot numbers, one has to show that $O \cap l(U_E) \neq \emptyset$. The following lemma asserts that this is indeed the case:

\begin{Lem*}
\label{simpel}
Let $L \subseteq \R^n$ be a cocompact lattice and $O \subseteq \R^n$ a nonempty open subset such that for all $v_1, v_2 \in O$ also $v_1 + v_2 \in O$. Then $O \cap L \neq \emptyset$. 
\end{Lem*}

\begin{proof}
Since $O$ is open and $L \otimes \Q$ is dense, there exists $x \in O \cap L \otimes \Q$. By taking $n x = x + \ldots + x$ for some $n \in \N_0$, we find $x \in L \cap O$.
\end{proof}

\noindent So this lemma implies that there exist unit Pisot numbers in every real Galois extension $E \neq \Q$. 
The lemma also implies that every open nonempty subset of $V$ which is invariant under addition gives rise to possible algebraic units. For example, there also exists unit Pisot numbers with an extra condition on them:

\begin{Prop*}
\label{existence}
Let $E$ be a real Galois extension of $\Q$ and fix some $\sigma \in \Gal(E,\Q)$ with $\sigma \neq 1$. Then there always exists an unit Pisot number $\lambda \in E$ such that $ \vert \sigma(\lambda^2) \lambda \vert < 1$. 
\end{Prop*}

\begin{proof}
Assume that $\sigma_2$ of the map $l$ given above is equal to $\sigma$. Let $O \subseteq V$ be the open nonempty subset of $V$ given by $x_1 > 0, x_i < 0$ for all $i \geq 2$ and $x_1 + 2 x_2 <0$, then there exists $x \in O \cap l(U_E)$ because of the previous lemma. Any element of the preimage of $x$ will satisfy the conditions of the proposition.
\end{proof}



\begin{center}
\textbf{Acknowledgments}
\end{center}
I would like to thank my advisor Karel Dekimpe for his useful comments on a first version of this article and the referee for his/her remarks which have improved this paper. 

\bibliography{G:/algebra/ref}
\bibliographystyle{G:/algebra/ref}

\end{document}